\documentclass[10pt]{article}

\usepackage{bigints}
\pdfoutput=1

\usepackage{amssymb}
\usepackage{mathrsfs}
\usepackage{amsmath}
\usepackage{amsfonts}

\usepackage{amsthm}
\newtheorem{theo}{Theorem}[section]
\newtheorem{prop}{Proposition}[section]
\newtheorem{coro}{Corollary}[section]

\theoremstyle{definition}

\theoremstyle{remark}

\newtheorem{rema}{Remark}[section]

\usepackage{hyperref}

\begin{document}

\title{\large{\textbf{Non-Einstein relative Yamabe metrics}}}

\author{Shota Hamanaka\thanks{supported in doctoral program in Chuo University, 2020.}}

\date{}

\maketitle

\begin{abstract}
In this paper, we give a sufficient condition for a positive constant scalar curvature metric on a manifold with boundary to be a relative Yamabe metric,
which is a natural relative version of the classical Yamabe metric. 
We also give examples of non-Einstein relative Yamabe metrics with positive scalar curvature.
\end{abstract}

\section{Introduction}
\footnote[0]{2010 {\it Mathematics Subject Classification.}~~Primary 53C; Secondary 57R, 58E. }
\footnote[0]{{\it Key words and phrases.}~~relative Yamabe metrics; constant scalar curvature metrics.}
~~Let $M$ be a compact connected smooth manifold of dimension $n \ge 3.$
Let $\mathscr{M}$ be the space of all Riemannian metrics on $M$ and 
$\mathcal{C}(M)$ the set of all conformal classes on $M.$
We consider the normalized Einstein-Hilbert functional $\mathcal{E}$ on the space $\mathscr{M}:$
\[
\mathcal{E}~:~\mathscr{M} \rightarrow \mathbb{R},~~g \mapsto \mathcal{E}(g) := \frac{\bigintss_{M} R_{g}dv_{g}}{\mathrm{Vol}_{g}(M)^{\frac{n-2}{n}}},
\]
where $R_{g},~dv_{g},~\mathrm{Vol}_{g}(M)$ denote respectively the scalar curvature of $g$, the volume measure of $g$ and the volume of $(M,g)$.
In the case of $\partial M = \emptyset,$ for $C \in \mathcal{C}(M),$
we define a number $Y(M,C) := \inf_{g \in C} \mathcal{E}(g)$ and it is called the \textit{Yamabe constant} of $C.$
And a metric $g \in C$ which achieves this infimum is called a \textit{Yamabe metric}
and has constant scalar curvature.
Yamabe, Trudinger, Aubin and Schoen have proved that any conformal class $C$ contains a Yamabe metric.
Conversely, a metric $g$ with constant scalar curvature is a Yamabe metric if either $R_{g} \le 0$
or $g$ is an Einstein metric(\cite{obata1962certain}, \cite{obata1971conjectures}).
On the other hand, Kato \cite{shin1994examples} gave a sufficient condition for a metric to be a Yamabe metric
and examples of non-Einstein Yamabe metrics with positive scalar curvature.

For the case of $\partial M \neq \emptyset,$
Yamabe metrics under minimal boundary condition have been studied (cf. \cite{akutagawa2001notes}, \cite{akutagawarelative}, \cite{escobar1992yamabe}).
For $\Bar{C} \in \mathcal{C}(M),$ we set $\Bar{C}_{0} := \bigl\{ g \in \Bar{C}~|~H_{g} = 0~\mathrm{on}~\partial M \bigr\},$
that is, the space of all \textit{relative metrics} in $\Bar{C},$
where $H_{g}$ denotes the mean curvature of $\partial M$ in $(M, g).$
A relative metric $g \in \Bar{C}_{0}$ is called a \textit{relative Yamabe metric}
if $g$ is a minimizer of $\mathcal{E}|_{\Bar{C}_{0}}.$
The infimum of $\mathcal{E}|_{\Bar{C}_{0}}$ is called the {\it relative Yamabe constant} of $\Bar{C},$
which is denoted by $Y(M, \partial M, \Bar{C}).$
Like the case of $\partial M = \emptyset,$
a relative metric $g$ with constant scalar curvature is a relative Yamabe metric if $R_{g} \le 0.$
Moreover, it is known that some Obata-type theorems hold under a suitable boundary condition (see \cite{akutagawa2}, \cite{escobar1990uniqueness} and Section 2 in this paper).

Our main result of this paper is a relative version of Kato's result in \cite{shin1994examples}.
More precisely, we will give the following sufficient condition for a relative metric of constant scalar curvature
to be a relative Yamabe metric,
and examples of non-Einstein relative Yamabe metrics with positive scalar curvature.
Our main result is the following:

\begin{theo}
\label{theo1}
Let $g$ be a relative Yamabe metric on a compact connected smooth manifold $M$
of dimension $n \ge 3$ with non-empty smooth boundary $\partial M$
with $R_{g} > 0$ on $M.$ 
Assume that $h$ is a relative metric on $M$ with constant scalar curvature
and that $\varphi$ is a diffeomorphism of $M$ such that
$dv_{\varphi^{*}h}=\gamma dv_{g}$ for some positive constant $\gamma$. 
If 
\begin{equation}
\label{eq:eq1}
R_{h}h \le R_{g}g,
\end{equation}
then $h$ is also a relative Yamabe metric.
Moreover, if 
\begin{equation}
\label{eq:eq2}
R_{h}h < R_{g}g,
\end{equation}
then $h$ is a unique relative Yamabe metric (up to positive constant) in the relative confomal class 
$[h]_{0}$ of $h$.
Here,
$[h]_{0} := \bigl\{  g \in [h] \bigm| H_{g} = 0~\mathrm{on}~\partial M \bigr\}
= \bigl\{ u^{\frac{4}{n - 2}} \cdot h \bigm| u \in C^{\infty}_{+}(M),~\nu_{h} (u) = 0~\mathrm{on}~\partial M \bigr\},$
where $\nu_{h}$ denotes the inward unit normal vector field of $\partial M$ with respect to $h$
in $M.$
\end{theo}

This paper organized as follows.
In Section 2, we recall some background materials and prove Theorem \ref{theo1}.
In Section 3, we give some examples of non-Einstein relative Yamabe metrics.

\section{Backgrounds and the proof of Theorem~\ref{theo1}}
~~Let $M$ be a compact connected smooth manifold of dimension $n \ge 3$ with non-empty smooth boundary $\partial M.$
As pointed out in {\cite[Lemma~2.1]{akutagawa2001notes}}, 
the relative Yamabe constant $Y(M, \partial M; \Bar{C})$ is characterized as follows:
\begin{prop}[{\cite[Lemma~2.1]{akutagawa2001notes}}]
\label{prop1}

For any fixed $g \in \Bar{C}_{0},$
\begin{equation*}
\begin{split}
Y \left( M, \partial M; \Bar{C} \right) &= \inf_{h \in \Bar{C}_0} \mathcal{E} (h) 
= \inf_{u \in C^{\infty}_{+}(M) ,~\nu_{g} (u) |_{\partial M}=0} \mathcal{E} (u^{\frac{4}{n - 2}} g) \\
&= \inf_{u \in C^{\infty}_{+}(M) ,~\nu_{g} (u) |_{\partial M}=0 } \frac{\bigintss_{M} \Bigl( \frac{4(n-1)}{n-2} \left| du \right|^{2}_{g} + R_{g}u^{2} \Bigr) dv_{g}}{\left( \bigintss_{M} u^{\frac{2n}{n-2}}dv_{g} \right)^{\frac{n-2}{n}}} \\
&= \inf_{f \in L^{1,2}\left( M \right) ,~f \not\equiv 0} \frac{\bigintss_{M} \Bigl( \frac{4(n-1)}{n-2} \left| df \right|^{2}_{g} + R_{g}f^{2} \Bigr) dv_{g}}{\left( \bigintss_{M} |f|^{\frac{2n}{n-2}}dv_{g} \right)^{\frac{n-2}{n}}},
\end{split}
\end{equation*}
where $L^{1,2}(M)$ denotes the Sobolev space of square-integrable functions on $M$ up to their first weak derivatives
(see \cite{aubin2013some} for example).
\end{prop}
\begin{rema}
From this proposition, the relative Yamabe constant $Y (M, \partial M; \Bar{C})$ coincides with the conformal invariant
$Q(M) = Q(M,\Bar{C})$ of $\Bar{C}$ (up to the positive factor $\frac{4(n-1)}{n-2}$) defined by Escobar \cite{escobar1992yamabe}.
\end{rema}
Each relative Yamabe metric $g \in \Bar{C}_{0}$ has constant scalar curvature $R_{g} = Y (M, \partial M; \Bar{C}) \cdot \mathrm{Vol}_{g}(M)^{-2/n}$ with $H_{g} = 0$
on $\partial M$.
Conversely, like the case of closed manifolds, 
a relative metric $g$ with $R_{g} = \mathrm{const}$ is a relative Yamabe metric if $R_{g} \le 0$.
In the case of closed manifolds, it is also known that a metric with constant scalar curvature is a Yamabe metric if it is an Einstein metric (see \cite{obata1962certain}, \cite{obata1971conjectures}).
On the other hand, in the case of $\partial M \neq \emptyset$, there is an Obata-type theorem for manifolds with \textit{totally~geodesic~boundary} by Escobar as follows:
\begin{theo}[{\cite[Theorems~3.2,~4.1]{escobar1990uniqueness}}]
\label{theo1.5}
Let $g$ be an Einstein metric of positive scalar curvature on a compact $n$-manifold $M$ with \textit{totally~geodesic~boundary}.
Then, for any constant scalar curvature relative metric $h \in [g]_{0}$ , the following uniqueness result holds~:

(1)~If $(M,[g])$ is conformally equivalent to $(S^{n}_{+},[g_{S}]),$ then there exist a homothety
$\Phi: (S^{n}_{+},g_{S}) \rightarrow (M,[g])$ and a conformal transformation $\varphi \in \mathrm{Conf}(S^{n}_{+},[g_{S}])$
such that $\Phi^{*} h = \varphi^{*}(\Phi^{*}g)$. Here, $g_{S}$ denotes the standard metric on $S^{n}_{+.}$

(2)~If $(M,[g])$ is not conformally equivalent to $(S^{n}_{+},[g_{S}])$ , then, up to homothety, $h = g$.
\end{theo}
On the other hand, in \cite{akutagawa2}, Akutagawa gave
a different rigidity theorem as follows (which is released from the assumption that the boundary is totally geodesic):
\begin{theo}[{\cite[Theorem~1.1]{akutagawa2}}]
\label{theo2}
Let $\bar{g}$ be a relative positive Einstein metric on a compact $n$-manifold $M$ with boundary.
Then, for any relative Einstein metric $\check{g} \in [\bar{g}]_{0},$
the following holds:

\noindent
(1)~Assume that $\bar{g}$ is a metric of positive constant curvature, and set $g := \bar{g}|_{\partial M}.$

(1.1)~If $(\partial M, [g])$ is conformally equivalent to $(S^{n-1}, [g_{S^{n-1}}]),$
then there exist

a homothety $\Phi : (S^{n}_{+}, g_{S}) \rightarrow (M, \bar{g})$ and a conformal transformation 

$\phi \in \mathrm{Conf}(S^{n}_{+}, [g_{S}])$
such that $\Phi^{*} \check{g} = \phi^{*}(\Phi^{*} \bar{g}).$

(1.2)~If $(\partial M, [g])$ is not conformally equivalent to $(S^{n-1}, [g_{S^{n-1}}]),$
then up to 

rescaling,
$\check{g} = \bar{g}.$

\noindent
(2)~If $\bar{g}$ is not a metric of positive constant curvature, then, up to rescaling, $\check{g} = \bar{g}.$
\end{theo}
For further details on the relative Yamabe problem, refer to 
\cite{akutagawa2001notes}, \cite{akutagawa2}, \cite{akutagawarelative}, \cite{escobar1990uniqueness} or \cite{escobar1992yamabe}.

In \cite{akutagawa2}, Akutagawa also gave an example of a metric which is not relative Einstein but relative Yamabe
(see {\cite[Counterexample]{akutagawa2}} or Section 3 of this paper). 
On the other hand, in Section 3, we will also give another examples
of such metrics.
In the following, we will give the proof of Theorem \ref{theo1},
which is similar to the proof of {\cite[Theorem]{shin1994examples}}.
\begin{proof}[Proof of Theorem~\ref{theo1}]
Since the scalar curvature is preserved under the pull-back action of diffeomorphisms, it is enough to consider the case that 
$\varphi = id_{M}$.

Let $\bar{h}=u^{\frac{4}{n-2}}h \in [h]$, $u \in C^{\infty}_{+}(M)$ with $\nu_{h}(u) |_{\partial M}=0$. 
From Proposition~\ref{prop1}, it is enough to prove for this situation.
Then
\[
\begin{split}
\mathcal{E} (\bar{h}) = \frac{\bigintss_{M} R_{\bar{h}}dv_{\bar{h}}}{Vol_{\bar{h}}(M)^{\frac{n-2}{n}}}
&= \frac{\bigintss_{M} u^{-\frac{n+2}{n-2}} \left( - \frac{4(n - 1)}{n - 2}\Delta_{h} u + R_{h}u \right)u^{\frac{2n}{n-2}}dv_{h}}{\left( \bigintss_{M} u^{\frac{2n}{n-2}}dv_{h} \right)^{\frac{n-2}{n}}} \\
&= \frac{\bigintss_{M} \left( -\frac{4(n-1)}{n-2} u\Delta_{h} u + R_{h}u^{2} \right)dv_{h}}{\left( \bigintss_{M} u^{\frac{2n}{n-2}}dv_{h} \right)^{\frac{n-2}{n}}},
\end{split}
\]
where $- \Delta_{h}$ denotes the non-negative Laplacian with respect to $h.$
Thus, using integration by parts and $\nu_{h}(u) |_{\partial M}=0$, we obtain
\begin{equation*}
\mathcal{E} (\bar{h}) = \frac{\bigintss_{M} \Bigl( \frac{4(n-1)}{n-2} \left| \nabla u \right|^{2}_{h} + R_{h}u^{2} \Bigr) dv_{h}}{\left( \bigintss_{M} u^{\frac{2n}{n-2}}dv_{h} \right)^{\frac{n-2}{n}}}.
\end{equation*}

Using $\frac{R_{h}}{R_{g}} \left| \nabla u \right|^{2}_{g} \le \left| \nabla u \right|^{2}_{h}$ (from~the assumption (\ref{eq:eq1}) in Theorem \ref{theo1}), we obtain
\begin{equation*}
\begin{split}
\mathcal{E} (\bar{h})&= \frac{\bigintss_{M} \left( \frac{4(n-1)}{n-2} \left| \nabla u \right|^{2}_{h} + R_{h}u^{2} \right) \gamma dv_{g}}{\left( \bigintss_{M} u^{\frac{2n}{n-2}} \gamma dv_{g} \right)^{\frac{n-2}{n}}} \\
&\ge \gamma^{1-\frac{n-2}{n}} \frac{R_{h}}{R_{g}} \frac{\bigintss_{M} \left( \frac{4(n-1)}{n-2} \left| \nabla u \right|^{2}_{g} + R_{g}u^{2} \right) dv_{g}}{\left( \bigintss_{M} u^{\frac{2n}{n-2}} dv_{g} \right)^{\frac{n-2}{n}}} \\
&= \gamma^{1-\frac{n-2}{n}} \frac{R_{h}}{R_{g}}  \mathcal{E} \left( u^{\frac{4}{n-2}}g \right).
\end{split}
\end{equation*}
Since $g$ is a relative Yamabe metric, $R_{g}>0$ and from (\ref{eq:eq1}),
\begin{equation}
\label{eq:eq3}
\mathcal{E} (\bar{h}) \ge \gamma^{1-\frac{n-2}{n}} \frac{R_{h}}{R_{g}} \mathcal{E} (g)
\end{equation}
\[
\begin{split}
= \gamma^{1-\frac{n-2}{n}} \frac{R_{h}}{R_{g}} \frac{\bigintss_{M} R_{g} dv_{g}}{\left( \bigintss_{M} dv_{g} \right)^{\frac{n-2}{n}}}
= \frac{\bigintss_{M} R_{h} \gamma dv_{g}}{\left( \bigintss_{M} \gamma dv_{g} \right)^{\frac{n-2}{n}}}
&= \frac{\bigintss_{M} R_{h} dv_{h}}{\left( \bigintss_{M} dv_{h} \right)^{\frac{n-2}{n}}} \\
&= \mathcal{E} (h).
\end{split}
\]
Therefore, by the definition of the relative Yamabe metrics, $h$ is a relative Yamabe metric on $M$.

In the above argument, if we assume that $\bar{h}$ is also a relative Yamabe metric,
then it must be $\mathcal{E} (\bar{h})=\mathcal{E} (h).$
In paticular, the inequality in (\ref{eq:eq3}) must be an equality.
Hence, if we assume (\ref{eq:eq2}) in Theorem \ref{theo1} in addition, then $\nabla u \equiv 0$ on $M$.
Therefore, $u \equiv \mathrm{const}$ on $M$ since $M$ is connected.
This means that $h$ is a unique relative Yamabe metric up to positive constant in $[h]_{0}$.
\end{proof}

We have a corollary for Theorem \ref{theo1}.

\begin{coro}
\label{coro1}
Let $M$ be the one as in Theorem~\ref{theo1}
and $\bigl\{ g_t \bigm| T \le t \le T' \bigr\}~(T < T')$ a smooth variation in $\mathscr{M}$ satisfying the following conditions:

(1) $R_{g_t} = \mathrm{const}$ on M and $H_{g_t} = 0$ on $\partial M$ for all $t \in [T,T']$,

(2) $g_{T}$ is a relative Yamabe metric,

(3) $R_{g_t} > 0$ on $M$ for all $t \in [T,T')$,

(4) $R_{g_{T'}}=0$ on $M$.

\noindent
Then there exists a positive constant $\delta > 0$ such that $g_t$ is also a relative Yamabe metric for every $t \in [T'- \delta, T')$.
\end{coro}

\begin{proof}
From {\cite[TH\'EOR\`EME]{banyaga1974}}, there exists a family $\bigl\{ \varphi_t \bigm| T \le t \le T' \bigr\}$ 
of diffeomorphisms of $M$ with $\varphi_{t} (\partial M) = \partial M$
such that $dv_{\varphi^{*}_{t} g_t} = \gamma_t dv_{g_T}$ for some $\gamma_{t} \in \mathbb{R}$
which is smooth with respect to $t$.
By the assumptions of the Corollary, $g_t$ satisfies the condition (\ref{eq:eq1}) in the Theorem~\ref{theo1} when $t$ is sufficiently close to $T'$.
Hence, we can apply Theorem~\ref{theo1} to such $g_t$ and then obtain Corollary \ref{coro1}.
\end{proof}

\section{Examples of non-Einstein relative Yamabe metrics}
In this section, we shall give some examples of non-Einstein relative Yamabe metrics.

\noindent
(1)(cf. {\cite[Example~2]{shin1994examples}})~~We consider the Berger sphere ;
\begin{equation*}
\begin{split}
 S^{3}(1) \cong SU(2) &= \Bigl\{ A \in M( 2 ; \mathbb{C} ) \Bigm| \mathrm{det}A = 1,~A^{*} = A^{-1} \Bigr\} \\
&= \Biggl\{ \begin{pmatrix} z & -\bar{w} \\ w & \bar{z} \end{pmatrix} \Biggm| (z,w) {\in \mathbb{C}}^{2}, |z|^{2} + |w|^{2} = 1 \Biggr\},
\end{split}
\end{equation*}
and set
\[ X_1 := \begin{pmatrix} \sqrt{-1} & 0 \\ 0 & -\sqrt{-1} \end{pmatrix},
X_2 := \begin{pmatrix} 0 & 1 \\ -1 & 0 \end{pmatrix},
X_3 := \begin{pmatrix} 0 & \sqrt{-1} \\ \sqrt{-1} & 0 \end{pmatrix}. \]
Then, this is a left invariant orthonormal frame with respect to the standard metric on $SU(2)$.
And we define a left invariant metric $g_{s,t} \left( 1 \le s \le t \right)$ on $SU(2)$ so that
\[ g_{s,t} \left( X_1,X_1 \right) = 1,~g_{s,t} \left( X_2,X_2 \right) = s, 
~g_{s,t} \left( X_3,X_3 \right) = t,~g_{s,t} \left( X_{i},X_{j} \right) = 0~~(i \neq j). \]
Then, 
$\Bigl\{ X_1,\frac{1}{\sqrt{s}} X_2,\frac{1}{\sqrt{t}} X_3 \Bigr\}$
is an orthonormal basis of $SU(2)$ with respect to $g_{s,t}$.
Using this basis, we can derive the scalar curvature of $g_{s,t}$ as
\begin{equation*}
 R_{g_{s,t}} = \frac{2}{st} \bigl\{ 2(s+t+st) - (1+s^{2}+t^{2}) \bigr\}.
\end{equation*}
And we define a subspace $SU(2)_{+}$ of $SU(2)$ so that
\[ SU(2)_{+} := \Biggl\{ \begin{pmatrix} z & -\bar{w} \\ w & \bar{z} \end{pmatrix} \Biggm| \mathrm{Im}z \ge 0 \Biggr\}. \]
Then, $\left( SU(2)_{+},g_{s,t} |_{SU(2)_{+}} \right)$ is a Riemannain manifold with boundary $\partial SU(2)_{+}$:
\[ \partial SU(2)_{+} = \Biggl\{ \begin{pmatrix} z & -\bar{w} \\ w & \bar{z} \end{pmatrix} \Biggm| \mathrm{Im}z = 0 \Biggr\}. \]
Then, $X_1$ forms a left invariant unit normal vector field of the boundary $\partial SU(2)_{+}$ and it is minimal with respect to $g_{s,t}$ (that is, $H_{g_{s,t}} = 0$ on $\partial SU(2)_{+}$).
Hence, from Theorem \ref{theo1}, $g_{s,t}$ is a relative Yamabe metric with constant positive scalar curvature, if $t \ge s+\sqrt{s}+1$ (since $g_{1,1}$ is the standard metric, therefore it is a relative Yamabe metric).

On the other hand, the Ricci curvature of $g_{s,t}$ can be calculated as follows:
\[ \mathrm{Ric}_{g_{s,t}}\left( X_1,X_1 \right) = -\frac{1}{st} \left( -2+2t^{2}+2s^{2}-4st \right), \]
\[ \mathrm{Ric}_{g_{s,t}}\left( \frac{X_2}{\sqrt{s}},\frac{X_2}{\sqrt{s}} \right) = -\frac{1}{st} \left( 2+2t^{2}-2s^{2}-4t \right), \]
\[ \mathrm{Ric}_{g_{s,t}}\left( \frac{X_3}{\sqrt{t}},\frac{X_3}{\sqrt{t}} \right) = -\frac{1}{st} \left( 2-2t^{2}+2s^{2}-4s \right), \]
\[ \mathrm{Ric}_{g_{s,t}}\left( X_{i},X_{j} \right) = 0~(i \neq j). \]
Hence, $g_{s,t}$ is an Einstein metric if and only if $s=t=1.$
Consequently, $g_{s,t}$ is a non-Einstein relative Yamabe metric with positive scalar curvature on $SU(2)_{+}$ if $t \ge s+\sqrt{s}+1.$
\\

\noindent
(2)(cf. {\cite[Counterexample]{akutagawa2}})~~Consider the Clifford torus $\Phi({T^{2}}):$
\[
\Phi : T^{2} = S^{1} \times S^{1} \rightarrow S^{3}(1) \subset \mathbb{C}^{2},~(\theta, \phi) \mapsto \frac{1}{\sqrt{2}}(e^{\sqrt{-1}\theta}, e^{\sqrt{-1}\phi})~(0 \le \theta, \phi \le 2 \pi). 
\]
Set 
\[
V_{1} \sqcup V_{2} = S^{3}(1) - \Phi(T^{2}),
\]
and let $g_{S}$ be the round metric of constant curvature one on the $3$-sphere $S^{3}.$
Let $\Bar{V}_{i}~(i = 1,2)$ be a solid torus with minimal boundary $\partial \Bar{V}_{i} = \Phi(T^{2}).$
Then $\bar{g} := g_{S}|_{\Bar{V}_{1}}$ is a metric of constant curvature one on $\Bar{V}_{1},$
and thus it is a relative Einstein metric.
And, there exists a relative Yamabe metric $\check{g} \in [\bar{g}]_{0}$
such that
\[
\mathcal{E} (\check{g}) = Y(\Bar{V}_{1}, \partial \Bar{V}_{1}, [\bar{g}])
< Y(S^{3}_{+}, S^{2}, [g_{S}]) = \mathcal{E} (\bar{g})
\]
(see {\cite[Counterexample]{akutagawa2}} for more detail).
Hence we have $\check{g} \neq \bar{g},$
and from Theorem \ref{theo2}, (1.2),
such $\check{g}$ is not an Einstein metric.

\begin{rema}
The boundary $\partial SU(2)_{+}$ in the above example (1) is not totally geodesic with respect to $g_{s,t}$.
On the other hand, Escobar shown the rigidity Theorem \ref{theo1.5} ({\cite[Theorems~3.2,~4.1]{escobar1990uniqueness}})
as mentioned above.
But, the same statement does not hold in general for Riemannian manifolds with {\it minimal boundary}.
In fact, there are some examples of metrics which are not relative Einstein, not conformally equivalent to 
the standard hemisphere but has constant scalar curvature (see {\cite[Counterexample]{akutagawa2}}, {\cite[p.~875]{escobar1990uniqueness}}).
These are counterexamples to (2) in Theorem \ref{theo1.5}.
\end{rema}

\subsection*{Acknowledgement}
I would like to thank my supervisor Kazuo Akutagawa for suggesting the initial direction for this study,
his good advice and support.

\bigskip
\noindent
\textit{E-mail address}:~a19.fg4w@g.chuo-u.ac.jp

\noindent
\textsc{Department Of Mathematics, Chuo University \\
1-13-27 Kasuga Bunkyo-ku, Tokyo 112-8551, Japan}

\end{document}